\documentclass[twoside]{amsart}

\numberwithin{equation}{section}
\usepackage{amsmath,amssymb,amsfonts,amsthm,latexsym}
\usepackage{graphicx,xcolor,enumerate}
\usepackage[margin=1in]{geometry}
\usepackage{longtable}
\usepackage{appendix}
\usepackage{xcolor}
\usepackage{stmaryrd}
\usepackage{mleftright}
\usepackage{breqn}
\usepackage{url}
\usepackage{orcidlink}

\newtheorem{theorem}{Theorem}[section]
\newtheorem{lemma}[theorem]{Lemma}
\newtheorem{proposition}[theorem]{Proposition}
\newtheorem{corollary}[theorem]{Corollary}

\theoremstyle{definition}

\newtheorem{example}[theorem]{Example}
\theoremstyle{remark}

\newcommand{\C}{\mathbb{C}}

\newcommand{\R}{\mathbb{R}}

\newcommand{\SO}{\operatorname{SO}}
\newcommand{\OO}{\operatorname{O}}

\newcommand{\bs}{\boldsymbol}

\begin{document}

\title{Hilbert measures on orbit spaces of coregular $\OO_m$-modules}

\author[H.-C.~Herbig]{Hans-Christian Herbig\,\orcidlink{0000-0003-2676-3340}}
\address{Departamento de Matem\'{a}tica Aplicada, Universidade Federal do Rio de Janeiro,
Av. Athos da Silveira Ramos 149, Centro de Tecnologia - Bloco C, CEP: 21941-909 - Rio de Janeiro, Brazil}
\email{herbighc@gmail.com}

\author[C. Seaton]{Christopher W. Seaton,\orcidlink{0000-0001-8545-2599}}
\address{Department of Mathematics and Statistics,
Skidmore College, 815 North Broadway,
Saratoga Springs, NY 12866, USA}
\email{cseaton@skidmore.edu}

\author[L. Whitesell]{Lillian Whitesell}
\address{Ohio State University,
Dept of Graduate Mathematics,
Mathematics Tower,
231 W 18th Ave,
Columbus OH 43210, USA}
\email{whitesell.lily@gmail.com}

\keywords{classical invariant theory, orthogonal group, change of variables, orbit space, Gram matrix}
\subjclass[2020]{primary 57S15;	secondary 14L30, 28A99}

\thanks{C.S. was supported by an AMS-Simons Research Enhancement Grant for PUI Faculty.
L.W. was supported by a Rhodes College Summer Research Fellowship, the Robert Allen Scott award from the Jane Hyde Scott Foundation and Rhodes College,
and the National Science Foundation under Grant No. 2015553.}

\begin{abstract}
We construct canonical measures, referred to as \textit{Hilbert measures}, on orbit spaces of classical coregular representations
of the orthogonal groups $\OO_m$. We observe that the measures have singularities along non-principal strata of the orbit space if
and only if the number of copies of the defining representation of $\OO_m$ is equal to $m$.
\end{abstract}

\maketitle
\tableofcontents
\section{Introduction}\label{sec:intro}

When using the change of variables for multidimensional integrals, the coordinate functions must be algebraically independent.
In the presence of a symmetry group acting on the integration domain it is natural to use fundamental invariants for the change of variables.
This type of consideration can be found in textbooks on multivariable calculus in the context of spherical coordinates.
For example, letting $\OO_m = \OO_m(\R)$ denote the $m\times m$ real orthogonal group,
one might want to integrate a smooth $\OO_{3}$-invariant function $ f:D\rightarrow \R$ defined on an
$\OO_3$-invariant domain $D\subseteq \R^{3}$. By the Theorem of Gerald W. Schwarz on differentiable invariants \cite[Theorem~1]{diffinv},
see also \cite[Theorem~1]{MatherDifferentiableInvariants}, one can write
$f( x_{1}, x_{2}, x_{3}) = F(u(x_1,x_2,x_3))$
for some function $F\colon I\mathbb{\rightarrow R}$ defined on the interval $u(D)\subset[0,\infty)$ where $u(x_1,x_2,x_3) = x_1^2 + x_2^2 + x_3^2$.
One easily derives
\[
    \int_{D} f(x_1, x_2, x_3)\;  dx_1 dx_2 dx_3
        = 2\pi \int_I F(u)\sqrt{u}\; du.
\]
We refer to $2\pi\sqrt{u}\; du$ as the \emph{Hilbert measure} on $[0,\infty)$ corresponding to the \emph{Hilbert basis} $\{u=x_1^2 + x_2^2 + x_3^2\}$.
A more conventional way to write this is $4\pi r^{2}\; dr$, where $r = \sqrt{u}$ is the radius function.
The interval $[0,\infty)$ can be interpreted as the diffeomorphic image of the orbit space $\R^3/\OO_{3}$ via the \emph{Hilbert embedding}
$u\colon\R^{3}\rightarrow\R$. In general, when  $G$ is a compact Lie group acting linearly on
a domain $D\subseteq\R^m$, a Hilbert embedding $\boldsymbol{u}=(u_1,\dots,u_k):D\mapsto \mathbb{R}^k$ comprised of a complete set of polynomial invariants $u_1,\dots,u_k$,  induces a map which embeds the orbit space $D/G$
of a $G$-invariant subset $D\subseteq\R^{m}$ into Euclidean space (see \cite[p.~18]{Bierstone}).

In this paper we deduce similar coordinate change formulas for the classical representations $V_{k} :=\R^{mk}$ of the orthogonal groups $\OO_{m}$,
where $\OO_{m}$ acts diagonally on $(\bs{v}_{1} ,\bs{v}_{2} ,\ldots ,\bs{v}_{k}) \in V_{k}$ with
$\bs{v}_{i} \in \R^{m}$ for each $i\in [k] = \{1,\ldots ,k\}$ (see, for example, \cite{CHS}). The condition that the fundamental invariants are algebraically independent is referred to as \emph{coregularity} of the representation $V_{k}$. It is well-known that $V_{k}$ is a coregular representation of $\OO_m(\C)$
if and only if $k\leq m$, see \cite[Theorem~11.18]{SchwarzLifting}, and the same holds for $\OO_m=\OO_m(\R)$ by
\cite[Proposition~5.8(1)]{SchwarzLiftingHomotopies}.
The fundamental polynomial invariants of the classical representation $V_{k}$ are given by the mutual inner products of the coordinate vectors.
That is, the invariants are given by
${u_{i,j} = u_{j,i} = \langle \bs{v}_{i} ,\bs{v}_{j} \rangle =\sum_{\ell =1}^{m}{v}_{i,\ell}{v}_{j,\ell}}$, see \cite[Section~9.3]{PopovVinberg}.
The Hilbert embedding is then the vector-valued map
\[
    \bs{u} =\left( u_{i,j} \right)_{1\leq j\leq i \leq k} \colon V_{k}\rightarrow \R^{\binom{k+1}{2}}
\]
whose components are given by the invariants. The image $ X=\bs{u}( V_{k})$ of $ \bs{u}$ is a semialgebraic subset of $ \R^{\binom{k+1}{2}}$,
see \cite[Theorem~(0.10)]{SchwarzProcesi}. In the coregular case when $ k\leq m$, the Zariski closure of
$ X$ is $ \R^{\binom{k+1}{2}}$, i.e., $ X$ can be described by polynomial inequalities only.
By general principles, $ \bs{u}$ descends to a diffeomorphism of the differential spaces $ \left( V_{k} /\OO_{m} ,\ \mathcal{C}^{\infty }( V_{k})^{\OO_{m}}\right)$ and $ \left( X,\ \mathcal{C}^{\infty }( X)\right)$. In particular, it restricts to a diffeomorphism between the isotropy type strata of $ V_{k} /\OO_{m}$ and the minimal semialgebraic strata of $X$, see \cite[Theorem~2.5]{Bierstone}.

Abusing notation, we use also $\bs{u} = (u_{i,j})$ with $u_{i,j} = u_{j,i}$ for the coordinates on
$U:=\R^{\binom{k+1}{2}}$ and use the notation
$\R[U] = \R[ u_{i,j} : 1\leq j\leq i\leq k]$ for the coordinate ring on $U$. The restriction $\left(\langle \bs{v}_{i} ,\bs{v}_{j} \rangle\right)_{i,j\in[k]}$ of the symmetric matrix $G_k = ( u_{i,j})_{i,j\in [k]}$ to $\bs{u}(V_k)$ is often referred to as the
\textit{Gram matrix}. The \textit{Hilbert measure} $\lambda_{k,m}(\bs{u})\; d\bs{u}$ on $\bs{u}(V_{k})$ is uniquely determined
by the condition that
\[
    \int_{V_{k}} f(\bs{v}_{1}, \bs{v}_{2}, \ldots, \bs{v}_{k})\; d\bs{v}_{1} d\bs{v}_{2} \cdots d\bs{v}_{k}
        =   \int_{\bs{u}( V_{k})} F(\bs u) \lambda_{k,m}(\bs{u}) d\bs{u}
\]
for each $\OO_{m}$-invariant function with compact support $f\in \mathcal{C}_{c}^{\infty }(V_{k})^{\OO_{m}}$. Here
$d\bs{v}_{i}$, $i=1,\ldots ,k$ is the
Lebesgue measure for the $ i$th copy of $ \R^{m}$ and $ f(\bs{v}_{1} ,\bs{v}_{2} ,\ldots ,\bs{v}_{m}) =F(\bs{u})$ for some
$F\in \mathcal{C}_{c}^{\infty }(\R^{\binom{m+1}{2}})$.

Now recall that the volume of the $(m-1)$-dimensional unit sphere $S^{m-1}$ is
$\operatorname{Vol}(S^{m-1}) = \frac{2\pi^{m/2}}{\Gamma(m/2)}$ and that
\begin{align}
\label{eq:zhang}
 \prod_{j=1}^{m}\operatorname{Vol}(S^{j})
    =   \prod_{j=1}^{m} \frac{2\pi^{j/2}}{\Gamma(j/2)}
    =   \frac{2^{m}\sqrt{\pi}^{\binom{m+1}{2}}}{\prod_{j=1}^{m} \Gamma (j/2)}
    =   \operatorname{Vol}(\OO_{m}),
 \end{align}
see \cite[Theorem~2.24]{ZhangVolumes}. The purpose of this paper is to prove the following result.

\begin{theorem}
\label{thrm:Main}
For $k\leq m$ the Hilbert measure on $\bs{u}(V_k)\simeq V_k/\OO_m$ is
\begin{align}\begin{split}
\label{eq:Main}
    \lambda_{k,m}(\bs{u}) \; d\bs{u}
        &=      \frac{|G_k|^{(m-k-1)/2}}{2^{k}}
                    \prod_{j=1}^{k}\operatorname{Vol}(S^{m-j})\; d\bs{u}
        \\&=    \frac{1}{2^{k}}
                    \operatorname{Vol}(\OO_{m}/\OO_{m-k})\;
                    |G_k|^{(m-k-1)/2} \; d\bs{u},
\end{split}\end{align}
where $\operatorname{Vol}(\OO_{m} /\OO_{m-k})$ is the volume of the Stiefel manifold
$\OO_{m}/\OO_{m-k}$ \cite[Proposition~2.23]{ZhangVolumes} and $G_k = ( u_{i,j})_{i,j\in [k]}$ is the Gram matrix.
\end{theorem}

Note that if we consider $V_k$ as an $\SO_m$-module, then $V_k$ is coregular if and only if $k\leq m-1$, see \cite[Theorem~11.18]{SchwarzLifting}.
In this case, the $\SO_m$- and $\OO_m$-invariants coincide, see \cite[Section~9.3]{PopovVinberg}, and each point in $V_k$ is fixed by a reflection so
that $V_k/\SO_m = V_k/\OO_m$. Hence we have following.

\begin{corollary}
\label{cor:SO}
For $k\leq m-1$ the Hilbert measure on $\bs{u}(V_k)\simeq V_k/\SO_m$ coincides with $\lambda_{k,m}(\bs{u}) \; d\bs{u}$ in Equation~\eqref{eq:Main}.
\end{corollary}

It is interesting to note that $\lambda_{k,m}$ is a smooth function on $\R^{\binom{k+1}{2}}$ in every case
except $k = m$, in which case it has singularities at points where $|G_k| = 0$. In terms of the stratification of $V_k$ by $\OO_m$-orbit types, the principal stratum consists of points such that $\{\bs{v}_1,\ldots,\bs{v}_k\}$ is linearly independent, i.e., such that $|G_k| \neq 0$. Hence, the singularities occur on lower-dimensional strata. Note that these lower-dimensional strata constitute a set of measure zero in $V_k$ so that the integral of a smooth invariant function, and hence the measure $\lambda_{k,m}\; d\bs{u}$, is determined by the principal stratum.

Using $\OO_{m}$-invariant bump functions localized near the orbits, see \cite[Theorem~4.3.1]{Palais} and \cite[Section~4.2]{Michor},
one deduces the following well-known result.

\begin{corollary}[{\cite[Example~0.8]{SchwarzProcesi}}]
\label{cor:HilbImage}
If $k\leq m$, then $\boldsymbol{u}(V_k)$ is the subset of $\R^{\binom{k+1}{2}}$ where $G_k$ is positive semidefinite.
\end{corollary}
\begin{proof}
Let us use induction on $m$ to show that the matrix $\left(\langle \bs{v}_{i} ,\bs{v}_{j} \rangle\right)_{i,j\in[k]}$  is positive semidefinite.
The statement is obvious for $m=1$. For $\ell< m$ we view $\mathbb{R}^{\ell}$ as the subspace
\[
    \left\{\boldsymbol{v} =( v_{1} ,\dotsc ,v_{m}) =\mathbb{R}^{m} \mid  v_{\mu} = 0\text{ for }\mu  >\ell\right\} \subset \mathbb{R}^{m} .
\]
From the theorem we can show, using invariant bump functions for points with principal isotropy type, that when $ k\leq m$ and $ k\neq m-1$ we have $ |G_k|_{|(V_{k})_{\operatorname{princ}}} \geq 0$. But
\begin{align*}
    (V_{k})_{\operatorname{princ}}
        &=  \{(\boldsymbol{v}_{1} ,\dotsc ,\boldsymbol{v}_{k}) \in V_{k} \mid \det( \langle \boldsymbol{v}_{i} ,\boldsymbol{v}_{j} \rangle )_{i,j\in [ k]}
            \ \neq 0\}  \\
        &=  \left\{(\boldsymbol{v}_{1} ,\dotsc ,\boldsymbol{v}_{k}) \in V_{k}  \mid
            \dim\operatorname{Span}\{\boldsymbol{v}_{1},\dotsc,\boldsymbol{v}_{k}\} =k\right\}.
\end{align*}
If $\dim\operatorname{Span}\mathbb{R}\{\boldsymbol{v}_{1} ,\dotsc ,\boldsymbol{v}_{m-1}\} =m-1$ we can find an $ A\in \operatorname{O}_{m}$ such that
$ ( A\boldsymbol{v}_{1} ,\dotsc ,A\boldsymbol{v}_{m-1}) \in \left(\mathbb{R}^{k( m-1)}\right)_{\operatorname{princ}}$. Hence we can show using the theorem for $ \mathbb{R}^{k( m-1)}$ that the upper left principal $ ( m-1) \times ( m-1)$-minor of $ ( \langle \boldsymbol{v}_{i} ,\boldsymbol{v}_{j} \rangle )_{i,j\in [ k]}$ is $ \geq 0$. So we can drop the assumption $ k\neq m-1$ above.

Now for $k\leq m$ we have $(\boldsymbol{v}_{1} ,\dotsc ,\boldsymbol{v}_{k}) \in V_{k} \backslash (V_{k})_{\operatorname{princ}}$ if and only if
$\det( \langle \boldsymbol{v}_{i} ,\boldsymbol{v}_{j} \rangle )_{i,j\in [ k]} =0$, or equivalently, if
$\ell=\dim\operatorname{Span}\{\boldsymbol{v}_{1} ,\dotsc ,\boldsymbol{v}_{k}\} < k$.
If $\ell < k$ we can find an $A\in \operatorname{O}_{m}$ such that
$(A\boldsymbol{v}_{1} ,\dotsc ,A\boldsymbol{v}_{k}) \in \left(\mathbb{R}^{k\ell}\right)_{\operatorname{princ}}$ and hence by the induction
hypothesis the upper left $\ell\times\ell$-submatrix of $(\langle \boldsymbol{v}_{i} ,\boldsymbol{v}_{j} \rangle )_{i,j\in [ k]}$ is positive semidefinite.

To argue that $\boldsymbol{u}$ reaches each point where $G_{k}$ is positive semidefinite, take a Cholesky decomposition $R^{\top } R$ of such a point $(u_{i,j})_{i,j\in [k]}$. Then $$(\boldsymbol{v}_{1} ,\dotsc ,\boldsymbol{v}_{k}) =\begin{pmatrix}
\boldsymbol{1}_{k}\\
\boldsymbol{0}_{m-k}
\end{pmatrix} R$$ is in the preimage.
\end{proof}

Let us elaborate on why Hilbert measures on orbit spaces may be of scientific interest. In Feynman's approach to quantum theory, one studies path integrals
\[
    I_\hbar=\int e^{\sqrt{-1} S/\hbar }\mathcal{D} \varphi,
\]
where $S$ is a local action functional (i.e., a space-time integral of a function on the infinite jet space in the fields
$\varphi_{\alpha} =\varphi_{\alpha }\left( x^{1} ,\dotsc ,x^{n}\right)$, $\alpha =1,\dotsc ,k$, depending on the spacetime coordinates $x^{1} ,\dotsc ,x^{n}$).
To date, only in low spacetime dimensions has it been possible to interpret the formal expression $\mathcal{D} \varphi$ as a well-defined measure.
In gauge theories, the actions $S$ and $\mathcal{D} \varphi$ are invariant with respect to the action of the infinite-dimensional gauge group $\mathcal{G}$.
Then the formally ill-defined integral $I_\hbar$ can be regularized by passing to an integral over the orbit space, that is, by integrating over the
coordinates perpendicular to the orbits and dividing out the (typically infinite) volume of the group. More specifically, it can be shown
\cite[Section 6]{Reshetikhin} that when the spacetime is $0$-dimensional and the $\mathcal{G}$-action has trivial principal isotropy type and admits
a cross-section (also known as gauge fixing), the \emph{Faddeev-Popov trick} \cite{Faddeev} provides a means to calculate the asymptotic expansion in the
Planck constant $\hbar$ of the regularized $I_\hbar$. The success of the Faddeev-Popov trick suggests that it might be valid in much greater generality.

The construction of the Hilbert measure that we carry out here for the coregular classical representations of the orthogonal groups can, in principle,
be performed in a similar and routine manner for other coregular representations of compact Lie groups. Coregular representations in
turn have been classified by Gerald W. Schwarz \cite{SchwarzCoregular}. In fact, coregular actions admit generically local cross-sections so that the invariants can
be complemented to local coordinate systems, and one has simply to work out the change of variables. In the non-coregular case things are more challenging
and we are examining several ideas. Besides, some of the authors have recently found symplectic forms $\omega$ on orbifolds of unitary representations
$V$ of finite groups \cite{HOS}. It seems to be a natural and important question whether their Liouville forms $\omega^{\cup \dim_{\mathbb C} V}$ can be related to the Hilbert measure. Similar questions arise for the singular symplectic quotients considered by Sjamaar-Lerman \cite{SL}.

\section*{Acknowledgements}

This paper developed from L.W.'s work as part of a Summer Research Fellowship and Jane Hyde Scott Award as well as continued research in
the Rhodes College Department of Mathematics and Statistics, and the authors gratefully acknowledge the support of the department and
college for these activities.
C.S. was supported by an AMS-Simons Research Enhancement Grant for PUI Faculty, and
L.W. was supported by the National Science Foundation under Grant No. 2015553.

\section{Generalized Euler Angles}
\label{sec:EA}

In this section, we recall the parametrization of $\operatorname{SO}_{m}$ by \textit{generalized Euler angles} given in \cite{HofRafRuedEulerAngles}.
This parametrization will be used implicitly in Section~\ref{subsec:CompCOB1}.
According to this parametrization, any $\bs{v} \in \R^{m}$ with $\| \bs{v} \| = 1$ can be written as
\begin{align}
\label{eq:sph}
 \bs{v} =\sin \theta_{1}\bs{e}_{1} +\cos \theta_{1}(\sin \theta_{2}\bs{e}_{2} +\cos \theta_{2}( \dots(\sin \theta_{m-1}\bs{e}_{m-1}
    + \cos \theta_{m-1}\bs{e}_{m}) \ldots ))
 \end{align}
for
\[
    \theta_{j} \in \left[ -\frac{\pi }{2} ,\frac{\pi }{2}\right],
    \quad
    j=1,\ldots ,m-2\ \text{ and } \theta_{m-1} \in [ -\pi ,\pi ].
\]
Here $\bs{e}_{j} \in \R^{m}$ is the $j$th canonical basis vector, i.e., its $k$th component is $\delta_{j,k}= 0$ if $j\neq k$ and $1$ if $j = k$.
The angle $\theta_{m-1}$ is referred to as the \textit{azimuthal angle }and the $\theta_{1} ,\ldots ,\theta_{m-2}$ as the \textit{generalized polar angles}. We write $\bs{\theta } =( \theta_{1} ,\ldots ,\theta_{m-1})$ for the $(m-1)$-tuple of Euler angles, which are generically uniquely determined by
$\bs{v}$.
\begin{example}
If $m=3$ we have $\bs{v} =\sin \theta_{1}\bs{e}_{1} +\cos \theta_{1}(\sin \theta_{2}\bs{e}_{2} +\cos \theta_{2}\bs{e}_{3})$.
This is can be understood in terms of the conventional spherical coordinates $(\phi,\theta)$ for the unit sphere $S^2\subset \mathbb R^3$ by setting
$\phi =\theta_{2}$ and $\theta = \pi /2 - \theta_{1}$ and applying the coordinate permutation $(13)$.
\end{example}
The vectors in the iterated brackets above, denoted $\bs{f}_k$, can be defined by recursive rotations, starting with $\bs{f}_{m} =\bs{e}_{m}$ and setting
\[
    \bs{f}_{k-1} = \sin \theta_{k-1}\bs{e}_{k-1} + \cos\theta_{k-1}\bs{f}_{k}, \quad  k=m,\ldots ,2,
\]
so that $\bs{v} =\bs{f}_{1}$. In fact, by construction $\langle \bs e_{k-1} ,\bs{f}_{k} \rangle =0$ for $k=2,\ldots ,m$,
$\theta_{k-1} =\measuredangle (\bs{f}_{k-1} ,\bs{f}_{k})$ for $k=1,\ldots, m-1$, and
$\| \bs{f}_{k} \| = 1$ for $k=1,\ldots ,m$. We can write
$\bs{f}_{k} = \sum_{j=1}^{k}\tan( \theta_{j})\left(\prod_{l=k}^{m}\cos(\theta_{l})\right)\bs{e}_{j}$.

\begin{proposition}{\cite{HofRafRuedEulerAngles}}
The vectors
$\bs{a}_{j}(\bs{\theta}) :=\left(\prod_{\ell=1}^{j-1}\cos( \theta_{\ell})\right)^{-1}\frac{\partial \bs{v}}
    {\partial \theta_{j}} =\cos( \theta_{j})\bs{e}_{j} -\sin( \theta_{j})\bs{f}_{j+1}$ with $j=1,\ldots ,m-1$
and $\bs{a}_{m} :=\bs{v}$ form an orthonormal basis of $\R^{m}$.
For $\bs{\theta}=(0,\dots,0)$ we have $\bs{a}_{j}(0,\dots,0)=\bs{e}_{j}$.
\end{proposition}
\begin{proof} From \eqref{eq:sph} it follows that
\[
    \frac{\partial \bs{v}}{\partial \theta_{j}} =\left(\prod_{\ell=1}^{j-1}\cos( \theta_{\ell})\right)\frac{\partial (\sin \theta_{j} \bs e_{j}
    + \cos( \theta_{j})\bs{f}_{j+1})}{\partial \theta_{j}} =\left(\prod_{\ell=1}^{j-1}\cos( \theta_{\ell})\right)(\cos \theta_{j}\bs e_{j}
    - \sin( \theta_{j})\bs{f}_{j+1}),
\]
so that $\bs{a}_{j} =\cos \theta_{j}\bs e_{j} - \sin(\theta_{j})\bs{f}_{j+1}$ is of magnitude $1$. Moreover, for $\ell< j$ we have
\begin{align}
\label{eq:par^2}
    \frac{\partial^{2}\bs{v}}{\partial \theta_{i} \partial \theta_{j}} =\frac{\partial \left(\prod_{\ell=1}^{j-1}\cos( \theta_{\ell})\right)}{\partial \theta_{i}}\frac{\partial (\sin \theta_{j}\bs e_{j} + \cos( \theta_{j})\bs{e}_{j+1})}{\partial \theta_{j}}
        =   -\tan \theta_{i}\frac{\partial \bs{v}}{\partial \theta_{j}}.
\end{align}
Since $\| \bs{v} \| =1$ \ we have $\langle \bs{v} ,\partial \bs{v} /\partial \theta_{j} \rangle =0$, so that $\bs{a}_{j} \perp \bs{a}_{m}$.
But
\[
    \left\langle \frac{\partial \bs{v}}{\partial \theta_{i}} ,\frac{\partial \bs{v}}{\partial \theta_{j}}\right\rangle
    =   \frac{\partial }{\partial \theta_{i}}\underbrace{\left\langle \bs{v} ,\frac{\partial \bs{v}}{\partial \theta_{j}}\right\rangle}_{=0}
        - \underbrace{\left\langle\bs{v} ,\frac{\partial^{2}\bs{v}}{\partial \theta_{i} \partial \theta_{j}}\right\rangle}_{=0\text{ by }\eqref{eq:par^2}}
    =   0,
\]
so that $\bs{a}_{i} \perp \bs{a}_{j}$ for $i,j=1,\ldots,m-1$.
\end{proof}

Let $A_{\bs{\theta }} :=( a_{i,j}(\bs{\theta }))_{i,j\in [m]} \in \operatorname{SO}_{m}$ be the matrix whose columns are $\bs{a}_{1} (\bs{\theta}),\ldots ,\bs{a}_{m}(\bs{\theta})$. It turns out that
\[a_{i,j}(\bs{\theta }) =\begin{cases}
\cos \theta_{i} & \text{if } i=j\leq m-1,\\
\tan \theta_{i}\prod_{\ell=1}^{i}\cos \theta_{\ell} & \text{if } i\leq m= j,\\
-\tan \theta_{i}\tan \theta_{j}\prod_{\ell=j}^{i}\cos \theta_{\ell} & \text{if } i > j,\\
0 & \text{else},
\end{cases}
\]
with the convention that $\theta_m=\pi/2$ and $\tan(\theta_m)\cos(\theta_m)=1$.

\begin{theorem}{\cite{HofRafRuedEulerAngles}}
Every matrix in $\SO_{m}$ can be written as an $A_{\bs{\theta}}$ for some generically unique $\bs{\theta} =(\theta_{1} ,\ldots ,\theta_{m-1})$.
\end{theorem}

\section{Computation of the measure}
\label{sec:MeasureComp}

In this section, we compute the measure $\lambda_{k,m}(\bs{u}) \; d\bs{u}$ and hence prove Theorem~\ref{thrm:Main}.
We divide the proof into two changes of variables.
The first finds a representative $(\bs{w}_1,\ldots,\bs{w}_k)$ of the $\OO_m$-orbit of a point $(\bs{v}_1,\ldots,\bs{v}_k)\in V_k$
in a specific fundamental domain $P$. The second transforms from $(\bs{w}_1,\ldots,\bs{w}_k)$ to the invariants $u_{i,j}$.

\subsection{First change of variables}
\label{subsec:CompCOB1}

Let $\bs{v}_1,\ldots,\bs{v}_k\in\R^m$ with coordinates $\bs{v}_i = (v_{i,1},v_{i,2},\ldots,v_{i,m})$ for each $i$.
As noted in the introduction, the measure $\lambda_{k,m}(\bs{u}) \; d\bs{u}$ is determined by the principal stratum
in terms of the stratification of $V_k$ by orbit types.
Hence, we may restrict our attention to the principal stratum and therefore assume $\{\bs{v}_1,\ldots,\bs{v}_k\}$ is linearly independent.
We first describe a change of variables to a representative $(\bs{w}_1,\ldots,\bs{w}_k)\in\R^{km}$ of the $\OO_m$-orbit $\OO_m(\bs{v}_1,\ldots,\bs{v}_k)\in\R^{km}$
contained in the fundamental domain $P$ where, using coordinates $\bs{w}_i = (w_{i,1},\ldots,w_{i,m})$ for the $\bs{w}_i$,
points in $P$ satisfy $w_{i,j} = 0$ for $j > i$ and $w_{i,i} \geq 0$ for each $i$; see Equation~\eqref{eq:COB1finalIntegral}.
Because the resulting $\bs{w}_i$ will form a linearly independent set, we will have $w_{i,i} > 0$ for each $i$.

Let $R_{\theta}^{(j)}\in\SO_m$ denote the element that acts on coordinates $j$ and $j+1$ as a rotation through the angle $\theta$
and acts trivially on all other coordinates. That is, $R_{\theta}^{(j)}$ acts on each $(v_{i,j},v_{i,j+1})$ as
\[
    R_{\theta}^{(j)} = \begin{pmatrix} \cos\theta & -\sin\theta \\ \sin\theta & \cos\theta \end{pmatrix}.
\]
We will define the $\bs{w}_i$ by successively applying rotations $R_{\theta}^{(j)}$ to make specific coordinates vanish following the approach
of \cite{HofRafRuedEulerAngles} recalled in Section~\ref{sec:EA}.

Let us illustrate the procedure with the case $m = k = 2$, which involves a single rotation and indicates the change of variables at
each step of the general case.
Let $v_{1,1}^\prime = \sqrt{v_{1,1}^2 + v_{1,2}^2}$, choose $\theta\in (-\pi,\pi]$ so that
\[
    R_{\theta}^{(1)}(v_{1,1}^\prime, 0) = (v_{1,1},v_{1,2})
    \quad\text{with }
    v_{1,1}^\prime > 0,
\]
and define $(v_{2,1}^\prime, v_{2,2}^\prime) = R_{-\theta}^{(1)}(v_{2,1},v_{2,2})$. Then
$(v_{1,1},v_{1,2}, v_{2,1}, v_{2,2})$ is given by
\[
    (v_{1,1}^\prime\cos\theta, v_{1,1}^\prime\sin\theta,
        v_{2,1}^\prime\cos\theta - v_{2,2}^\prime\sin\theta,
        v_{2,1}^\prime\sin\theta + v_{2,2}^\prime\cos\theta),
\]
and the Jacobian of the change of variables is given by
\[
    \frac{\partial (v_{1,1},v_{1,2}, v_{2,1}, v_{2,2})}{\partial (v_{1,1}^\prime,\theta,v_{2,1}^\prime,v_{2,2}^\prime)}
    =
    \begin{pmatrix}
        \cos\theta    &   - v_{1,1}^\prime\sin\theta   &   0   &   0   \\
        \sin\theta    &     v_{1,1}^\prime\cos\theta   &   0   &   0   \\
        0 & - v_{2,1}^\prime\sin\theta - v_{2,2}^\prime\cos\theta & \cos\theta & -\sin\theta
        \\
        0 & v_{2,1}^\prime\cos\theta - v_{2,2}^\prime\sin\theta & \sin\theta & \cos\theta
    \end{pmatrix}.
\]
Then we have
\[
    dv_{1,1}\wedge dv_{1,2}\wedge dv_{2,1}\wedge dv_{2,2}
        =   v_{1,1}^\prime
        dv_{1,1}^\prime\wedge d\theta\wedge dv_{2,1}^\prime\wedge dv_{2,2}^\prime.
\]

Now, we consider the general case of arbitrary $k \leq m$. Let us first explain our indexing convention.
Subscripts will continue to index coordinates of the $k$ vectors $\bs{v}_i$ as above. We will apply a sequence
of rotations, each targeting a specific vector $\bs{v}_i$ to make a specific coordinate vanish, and superscripts
will index this sequence. The angle $\theta^{(i,j)}$ is the $j$th rotation targeting the $i$th vector; after
its application, the $(m-j+1)$st coordinate of the $i$th vector will vanish. The coordinates after the application
the rotation through $\theta^{(i,j)}$ will be denoted $v_{k,\ell}^{(i,j)}$.
Note that we order the variables lexicographically, i.e., $v_{1,1}, v_{1,2}, \ldots, v_{1,m}, v_{2,1},\ldots,v_{k,m}$,
and similarly for the $v_{k,\ell}^{(i,j)}$.

In the first step, we define a change of variables $(v_{i,j})\mapsto (v_{i,j}^{(1,1)})$ such that $v_{1,m}^{(1,1)} = 0$
and $v_{1,m-1}^{(1,1)} > 0$. This is accomplished as in the case of $k=m=2$ with a rotation
$R_{\theta^{(1,1)}}^{(m-1)}$ with $\theta^{(1,1)}\in(-\pi,\pi]$, so that $v_{i,j}^{(1,1)} = v_{i,j}$ for all $i,j$ with
$j < m - 1$. The resulting change of variables is computed similarly to the case $k=m=2$ and is given by
\[
    \bigwedge\limits_{\substack{1\leq i\leq k \\ 1\leq j \leq m}} dv_{i,j}
        =   v_{1,m-1}^{(1,1)} \bigwedge\limits_{j=1}^{m-1} dv_{1,j}^{(1,1)} \wedge d\theta^{(1,1)}\wedge
        \bigwedge\limits_{\substack{2\leq i\leq k \\ 1\leq j \leq m}} dv_{i,j}^{(1,1)}.
\]

For the next step, we define $(v_{i,j}^{(1,1)})\mapsto (v_{i,j}^{(1,2)})$ such that $v_{1,m}^{(1,2)} = v_{1,m-1}^{(1,2)} = 0$
and $v_{1,m-2}^{(1,2)} > 0$, with $v_{i,j}^{(1,2)} = v_{i,j}^{(1,1)}$ for $j < m-2$.
Here, we apply a rotation $R_{\theta^{(1,2)}}^{(m-2)}$ chosen so that $R_{-\theta^{(1,2)}}^{(m-2)}$ rotates
$(v_{1,m-2}^{(1,1)},v_{1,m-1}^{(1,1)})$ to the positive $v_{m-2}$-axis; as $v_{1,m-1}^{(1,1)} > 0$, we have $\theta^{(1,2)}\in(0,\pi)$.
Then
\[
    \bigwedge\limits_{j=1}^{m-1} dv_{1,j}^{(1,1)} \wedge d\theta^{(1,1)}\wedge
        \bigwedge\limits_{\substack{2\leq i\leq k \\ 1\leq j \leq m}} dv_{i,j}^{(1,1)}
        =
        v_{1,m-2}^{(1,2)} \bigwedge\limits_{j=1}^{m-2} dv_{1,j}^{(1,2)} \wedge d\theta^{(1,2)} \wedge d\theta^{(1,1)}\wedge
            \bigwedge\limits_{\substack{2\leq i\leq k \\ 1\leq j \leq m}} dv_{i,j}^{(1,2)}
\]
so that
\begin{equation}
\label{eq:COB1Step2PreSub}
    \bigwedge\limits_{\substack{1\leq i\leq k \\ 1\leq j \leq m}} dv_{i,j}
        =   v_{1,m-1}^{(1,1)} v_{1,m-2}^{(1,2)} \bigwedge\limits_{j=1}^{m-2} dv_{1,j}^{(1,2)} \wedge d\theta^{(1,2)} \wedge d\theta^{(1,1)}\wedge
            \bigwedge\limits_{\substack{2\leq i\leq k \\ 1\leq j \leq m}} dv_{i,j}^{(1,2)}.
\end{equation}
However, note that
\[
    v_{1,m-2}^{(1,2)} = \sqrt{(v_{1,m-2}^{(1,1)})^2 + (v_{1,m-1}^{(1,1)})^2},
\]
hence
\begin{align*}
    v_{1,m-1}^{(1,1)}
        &=      \sqrt{(v_{1,m-2}^{(1,2)})^2 - (v_{1,m-2}^{(1,1)})^2}
        \\&=    \sqrt{(v_{1,m-2}^{(1,2)})^2 - (v_{1,m-2}^{(1,2)}\cos\theta^{(1,2)})^2}
        \\&=    v_{1,m-2}^{(1,2)}\sin\theta^{(1,2)}.
\end{align*}
Applying this substitution to Equation~\eqref{eq:COB1Step2PreSub} yields
\begin{equation}
\label{eq:COB1Step2}
    \bigwedge\limits_{\substack{1\leq i\leq k \\ 1\leq j \leq m}} dv_{i,j}
        =   (v_{1,m-2}^{(1,2)})^2\sin\theta^{(1,2)}
            \bigwedge\limits_{j=1}^{m-2} dv_{1,j}^{(1,2)} \wedge d\theta^{(1,2)} \wedge d\theta^{(1,1)}\wedge
            \bigwedge\limits_{\substack{2\leq i\leq k \\ 1\leq j \leq m}} dv_{i,j}^{(1,2)}.
\end{equation}

To consider one more step in detail, define $(v_{i,j}^{(1,2)})\mapsto (v_{i,j}^{(1,3)})$ such that
$v_{1,m}^{(1,3)} = v_{1,m-1}^{(1,3)} = v_{1,m-2}^{(1,3)} = 0$
and $v_{1,m-3}^{(1,3)} > 0$. We choose $\theta^{(1,3)}\in(0,\pi)$ so that $R_{-\theta^{(1,3)}}^{(m-3)}$ rotates
$(v_{1,m-3}^{(1,2)}, v_{1,m-2}^{(1,2)})$ to the positive $v_{m-3}$-axis, yielding the following change of variables from
Equation~\eqref{eq:COB1Step2}:
\begin{align}
    \nonumber
    &(v_{1,m-2}^{(1,2)})^2\sin\theta^{(1,2)}
        \bigwedge\limits_{j=1}^{m-2} dv_{1,j}^{(1,2)} \wedge d\theta^{(1,2)} \wedge d\theta^{(1,1)}\wedge
        \bigwedge\limits_{\substack{2\leq i\leq k \\ 1\leq j \leq m}} dv_{i,j}^{(1,2)}
    \\\label{eq:COB1Step3PreSub}
    &\qquad\qquad=
        (v_{1,m-2}^{(1,2)})^2 v_{1,m-3}^{(1,3)} \sin\theta^{(1,2)}
            \bigwedge\limits_{j=1}^{m-3} dv_{1,j}^{(1,3)} \wedge d\theta^{(1,3)} \wedge d\theta^{(1,2)} \wedge d\theta^{(1,1)}\wedge
            \bigwedge\limits_{\substack{2\leq i\leq k \\ 1\leq j \leq m}} dv_{i,j}^{(1,3)}.
\end{align}
Once again,
\[
    v_{1,m-3}^{(1,3)} = \sqrt{(v_{1,m-3}^{(1,2)})^2 + (v_{1,m-2}^{(1,2)})^2},
\]
and $v_{1,m-3}^{(1,2)} = v_{1,m-3}^{(1,3)}\cos\theta^{(1,3)}$ so that
\[
    v_{1,m-2}^{(1,2)}
        =       v_{1,m-3}^{(1,3)}\sin\theta^{(1,3)},
\]
and Equation~\eqref{eq:COB1Step3PreSub} becomes
\[
    \bigwedge\limits_{\substack{1\leq i\leq k \\ 1\leq j \leq m}} dv_{i,j}
        =   (v_{1,m-3}^{(1,3)})^3 \sin\theta^{(1,2)} \sin^2\theta^{(1,3)}
            \bigwedge\limits_{j=1}^{m-3} dv_{1,j}^{(1,3)} \wedge d\theta^{(1,3)} \wedge d\theta^{(1,2)} \wedge d\theta^{(1,1)}\wedge
            \bigwedge\limits_{\substack{2\leq i\leq k \\ 1\leq j \leq m}} dv_{i,j}^{(1,3)}.
\]

Continuing in this way, we apply $R_{\theta^{(1,j)}}^{(m-j)}$ for $j = 4,5,\ldots,m-1$ to obtain coordinates $v_{i,j}^{(1,m-1)}$ such that
$v_{1,1}^{(1,m-1)} > 0$ and $v_{1,j}^{(1,m-1)} = 0$ for $j > 1$. Identical computations to those above at each step yield
\[
    \bigwedge\limits_{\substack{1\leq i\leq k \\ 1\leq j \leq m}} dv_{i,j}
        =   (v_{1,1}^{(1,m-1)})^{m-1}
            \prod\limits_{j=2}^{m-1} \sin^{j-1}\theta^{(1,j)} \;
            dv_{1,1}^{(1,m-1)} \wedge \bigwedge\limits_{j=1}^{m-1} d\theta^{(1,m-j)} \wedge
            \bigwedge\limits_{\substack{2\leq i\leq k \\ 1\leq j \leq m}} dv_{i,j}^{(1,m-1)}.
\]

We then apply the same process to target $\bs{v}_2$. The first rotation $R_{\theta^{(2,1)}}^{(m-1)}$ yields coordinates
$v_{i,j}^{(2,1)}$ such that $v_{2,m}^{(2,1)} = 0$ and $v_{2,m-1}^{(2,1)} > 0$, the second $R_{\theta^{(2,2)}}^{(m-2)}$
yields $v_{i,j}^{(2,2)}$ where the last two coordinates of $\bs{v}_2$ vanish, etc., ending with $R_{\theta^{(2,m-2)}}^{(2)}$,
yielding $v_{i,j}^{(2,m-2)}$ such that $v_{2,2}^{(2,m-2)} > 0$ and $v_{2,j}^{(2,m-2)} = 0$ for $j\geq 3$,
with the $v_{1,i}^{(1,m-1)} = v_{1,i}^{(2,m-2)}$ unchanged.
Targeting $\bs{v}_3$, we apply $m-3$ rotations so that the first three resulting coordinates are nonzero, etc., finally ending
with coordinates $v_{i,j}^{(k,m-k)}$ such that $v_{i,j}^{(k,m-k)} = 0$ whenever $j > i$ and $v_{i,i}^{(k,m-k)} > 0$ for each $i$.
If $k < m$, then the vectors $(\bs{v}_1,\ldots, \bs{v}_{k})$ are fixed by a reflection so that their $\SO_m$- and $\OO_m$-orbits coincide;
if $k = m$, then in the final step, we apply a reflection to ensure that $v_{k,k}^{(k,m-k)} > 0$, introducing a factor of
$2 = \delta_{k,m} + 1$. Identical computations to those above and a little book-keeping yields
\begin{equation}
\label{eq:COB1final}
    \bigwedge\limits_{\substack{1\leq i\leq k \\ 1\leq j \leq m}} dv_{i,j}
        =   (\delta_{k,m}+1)\prod\limits_{i=1}^{k} (v_{i,i}^{(k,m-k)})^{m-i}
            \prod\limits_{i=1}^{k} \prod\limits_{j=2}^{m-i} \sin^{j-1}\theta^{(i,j)}
            \\
            \bigwedge\limits_{i=1}^{k}\left(
                \bigwedge\limits_{j=1}^{i} dv_{i,j}^{(k,m-k)} \wedge
                \bigwedge\limits_{\ell=i}^{m-1} d\theta^{(i,m-\ell)}
            \right).
\end{equation}
Set $w_{i,j} = v_{i,j}^{(k,m-k)}$ for $1\leq j\leq i\leq k$, let
\[
    W = \{ w_{i,j} \in\R \mid
        1\leq j\leq i\leq k,
        w_{i,i} > 0
        \}
\]
denote the interior of the fundamental domain $P$ consisting of linearly independent sets of vectors,
and let
\[
    \Theta = \{ \theta^{(\ell,p)} \in(-\pi,\pi) \mid
        1\leq\ell\leq k,
        1\leq p\leq m-\ell,
        \theta_{\ell,p}\in(0,\pi)\text{ for }p > 1
        \}.
\]
Now observe that, for a smooth invariant integrable function $f$ on $V_k$, we have that $f(\bs{v}) = f(\bs{w})$, so we can express
\begin{align}
    \nonumber
    \int\limits_{V_k} f(\bs{v}) \bigwedge\limits_{\substack{1\leq i\leq k \\ 1\leq j \leq m}} dv_{i,j}
        &=      (\delta_{k,m}+1) \int\limits_{W\times\Theta} f(\bs{v})
                \prod\limits_{i=1}^{k} w_{i,i}^{m-i}
                \prod\limits_{i=1}^{k} \prod\limits_{j=2}^{m-i} \sin^{j-1}\theta^{(i,j)}
                \bigwedge\limits_{i=1}^{k}\left(
                    \bigwedge\limits_{j=1}^{i} dw_{i,j} \wedge
                    \bigwedge\limits_{\ell=i}^{m-1} d\theta^{(i,m-\ell)}
                \right)
        \\&=    \label{eq:COB1Integral}
                (\delta_{k,m}+1) \int\limits_W f(\bs{w}) \prod\limits_{i=1}^{k} w_{i,i}^{m-i}
                    \bigwedge\limits_{i=1}^{k} \bigwedge\limits_{j=1}^{i} dw_{i,j}
                \int\limits_\Theta
                    \prod\limits_{i=1}^{k} \prod\limits_{j=2}^{m-i} \sin^{j-1}\theta^{(i,j)}
                \bigwedge\limits_{i=1}^{k} \bigwedge\limits_{\ell=i}^{m-1} d\theta^{(i,m-\ell)}.
\end{align}
The integral over $\Theta$ can be computed using
\[
    \int\limits_0^\pi \sin^\ell \theta\;d\theta
        =   \frac{\sqrt{\pi}\;\Gamma\left(\frac{\ell+1}{2}\right)}{\Gamma\left(\frac{\ell+2}{2}\right)},
\]
see \cite[6.2.1 and 6.2.2]{AbramowitzStegun}, and we obtain
\begin{align*}
    \int\limits_\Theta \prod\limits_{i=1}^{k} \prod\limits_{j=2}^{m-i} \sin^{j-1}\theta^{(i,j)}
        \bigwedge\limits_{i=1}^{k} \bigwedge\limits_{\ell=i}^{m-1} d\theta^{(i,m-\ell)}
    &=
    \prod\limits_{i=1}^{\min\{k,m-1\}} 2\pi \prod\limits_{j=2}^{m-i}
        \frac{\sqrt{\pi}\;\Gamma\left(\frac{j}{2}\right)}{\Gamma\left(\frac{j+1}{2}\right)}
    \\&=
    \prod\limits_{i=1}^{\min\{k,m-1\}}
        \frac{ 2 \pi^{(m-i+1)/2} }{\Gamma\left(\frac{m-i+1}{2}\right)}
    \\&=
    \prod\limits_{i=1}^{\min\{k,m-1\}} \operatorname{Vol}( S^{m-i}).
\end{align*}
Noting that in the case $k = m$, the factor $\delta_{k,m} + 1 = \operatorname{Vol}(S^0)$,
Equation~\eqref{eq:COB1Integral} becomes
\begin{equation}
\label{eq:COB1finalIntegral}
    \int\limits_{V_k} f(\bs{v}) \bigwedge\limits_{\substack{1\leq i\leq k \\ 1\leq j \leq m}} dv_{i,j}
        =       \prod\limits_{i=1}^{k} \operatorname{Vol}\left( S^{m-i}\right)
                    \int\limits_W f(\bs{w}) \prod\limits_{i=1}^{k} w_{i,i}^{m-i}
                        \bigwedge\limits_{i=1}^{k} \bigwedge\limits_{j=1}^{i} dw_{i,j}.
\end{equation}

\subsection{Second change of variables}
\label{subsec:CompCOB2}

Our final task is to express the integral
\begin{equation}
\label{eq:COB2Initial}
    \int\limits_W f(\bs{w}) \prod\limits_{i=1}^{k} w_{i,i}^{m-i}
                        \bigwedge\limits_{i=1}^{k} \bigwedge\limits_{j=1}^{i} dw_{i,j}
\end{equation}
from Equation~\eqref{eq:COB1finalIntegral} in terms of the invariants $\left( u_{i,j} \right)_{1\leq j\leq i \leq k}$.
Recall that
\begin{align*}
    \bs{w}_1    &=      (w_{1,1},0,\ldots,0),
    \\
    \bs{w}_2    &=      (w_{2,1},w_{2,2},0,\ldots,0),
    \\
    \vdots
    \\
    \bs{w}_k    &=      (w_{k,1},w_{k,2},\ldots,w_{k,k},0,\ldots,0),
\end{align*}
and $w_{i,i} > 0$ for each $i=1,\ldots,k$.
We continue to use lexicographic ordering of the variables, i.e., ${w_{1,1}, w_{2,1}, w_{2,2},\ldots,w_{k,k}}$ and similarly for the $u_{i,j}$.
As in Section~\ref{sec:intro}, let $G_\ell = (\langle\bs{w}_i,\bs{w}_j\rangle)_{i,j\in[\ell]}$ denote the Gram matrix of the
set $\{\bs{w}_1,\ldots,\bs{w}_\ell\}$ of the first $\ell$ vectors. As usual, we use $|G_\ell|$ to denote the determinant of the
Gram matrix and set $|G_0| = 1$. We begin by computing the diagonal coordinates $w_{i,i}$ in terms of the invariants.

\begin{lemma}
\label{lem:COB2Gram}
With the $\bs{w}_i$ and $G_\ell$ as above, we have
\[
    w_{i,i} =   \sqrt{\frac{|G_i|}{|G_{i-1}|}}
\]
for $i = 1,2,\ldots,k$.
\end{lemma}
\begin{proof}
We argue by induction on $i$. If $i = 1$, then
\[
    w_{1,1}^2   =   \langle \bs{w}_1,\bs{w}_1\rangle    =   |G_1|
\]
so that
\[
    w_{1,1}     =   \sqrt{\frac{|G_1|}{|G_0|}}.
\]
So assume there is an $i > 1$ such that
\[
    w_{j,j}   =   \sqrt{\frac{|G_{j}|}{|G_{j-1}|}}
\]
for $1\leq j < i$. Let
\[
    \bs{W}_i =   \begin{pmatrix}
                    w_{1,1} &   0       &   \cdots  &   0   \\
                    w_{2,1} &   w_{2,2} &   \cdots  &   0   \\
                    \vdots  &           &   \ddots  &   \vdots  \\
                    w_{k,1} &   w_{k,2} &   \cdots  &   w_{k,k}
                    \end{pmatrix},
\]
and then as the $\bs{w}_j$ are the columns of $\bs{W}_i$ extended by $0$, we have $G_i = \bs{W}_i\bs{W}_i^T$
so that $|G_i| = |\bs{W}_i|^2$. Then
\begin{align*}
    \det\bs{W}_i
        &=      \prod\limits_{j=1}^i w_{j,j}
        \\&=    w_{i,i} \prod\limits_{j=1}^{i-1} \sqrt{\frac{|G_j|}{|G_{j-1}|}}
        \\&=    w_{i,i}\sqrt{|G_{i-1}|}.
\end{align*}
and then $|G_i| = |\det\bs{W}_i|^2 = w_{i,i}^2 |G_{i-1}|$, completing the proof.
\end{proof}

Using Lemma~\ref{lem:COB2Gram}, the product in Equation~\eqref{eq:COB2Initial} is telescoping,
\[
    \prod\limits_{i=1}^k w_{i,i}^{m-i}
    =
    \prod\limits_{i=1}^k \frac{|G_i|^{(m-i)/2} }{|G_{i-1}|^{(m-i)/2} }
    =
    |G_k|^{(m-k)/2} \prod\limits_{i=1}^{k-1} \sqrt{ |G_i| },
\]
so we can rewrite Equation~\eqref{eq:COB2Initial} as
\begin{equation}
\label{eq:COB2ProductGram}
    \int\limits_W f(\bs{w}) |G_k|^{(m-k)/2} \prod\limits_{i=1}^{k-1} \sqrt{ |G_i| }
        \bigwedge\limits_{i=1}^{k} \bigwedge\limits_{j=1}^{i} dw_{i,j}.
\end{equation}
We now consider the change of variables expressing the $dw_{i,j}$ in terms of the invariants.

\begin{lemma}
\label{lem:COB2Jacobian}
Ordering the variables lexicographically as above and choosing the representative $u_{i,j}$ such that $j \leq i$,
the Jacobian $\frac{\partial\bs{w}}{\partial\bs{u}}$ is lower-triangular with determinant given by
\[
    \left|\frac{\partial\bs{w}}{\partial\bs{u}}\right|
    =
    \frac{1}{2^k \sqrt{|G_1||G_2|\cdots|G_k|} }.
\]
\end{lemma}
\begin{proof}
Because $u_{i,j}$ for $j\leq i$ depends only on $w_{r,s}$ with $r\leq i$ and $s\leq j$, we have that $\frac{\partial\bs{u}}{\partial\bs{w}}$
and hence its inverse $\frac{\partial\bs{w}}{\partial\bs{u}}$ is lower-triangular. The diagonal entries of
$\frac{\partial\bs{u}}{\partial\bs{w}}$ are
\[
    \frac{\partial u_{i,j}}{\partial w_{i,j}}
    =
    (1 + \delta_{i,j})w_{j,j},
\]
so that, applying Lemma~\ref{lem:COB2Gram},
\begin{align*}
    \left|\frac{\partial\bs{u}}{\partial\bs{w}}\right|
        &=      2^k\prod\limits_{i=1}^k\prod\limits_{j=1}^i w_{j,j}
        \\&=    2^k\prod\limits_{i=1}^k\prod\limits_{j=1}^i \sqrt{\frac{|G_j|}{|G_{j-1}|}}
        \\&=    2^k\sqrt{G_1 G_2 \cdots G_k}.
        \qedhere
\end{align*}
\end{proof}

Letting $F(\bs{u})$ be a smooth function on $X$ such that
$F(\bs{u}) = f(\bs{v})$ and applying Lemma~\ref{lem:COB2Jacobian}, Equation~\eqref{eq:COB2ProductGram} becomes
\begin{equation}
\label{eq:COB2ProductRemoved}
    \frac{1}{2^k}\int\limits_W F(\bs{u}) |G_k|^{(m-k-1)/2}
        \bigwedge\limits_{i=1}^{k} \bigwedge\limits_{j=1}^{i} du_{i,j}.
\end{equation}
Recalling that $X = \bs{u}(V_k)$ denotes the image of the Hilbert embedding,
Equation~\eqref{eq:COB1finalIntegral} can now be expressed as
\begin{equation}
\label{eq:FinalIntegral}
    \int\limits_{V_k} f(\bs{v}) \bigwedge\limits_{\substack{1\leq i\leq k \\ 1\leq j \leq m}} dv_{i,j}
        =       \frac{1}{2^{k}}\prod\limits_{i=1}^{k} \operatorname{Vol}\left( S^{m-i}\right)
                    \int\limits_X F(\bs{u}) |G_k|^{(m-k-1)/2}
                        \bigwedge\limits_{i=1}^{k} \bigwedge\limits_{j=1}^{i} du_{i,j},
\end{equation}
which completes the proof of Theorem~\ref{thrm:Main}.

\section{Examples}
\label{sec:Ex}

\subsection{Two particles in $\R^3$}
\label{subsec:Ex2inR3}

As a concrete illustration,
let us consider the example of two particles in $\R^3$, i.e., $k = 2$ and $m = 3$.
Let $\bs{v}_1 = (v_{11}, v_{12}, v_{13})$ and $\bs{v}_2 = (v_{21}, v_{22}, v_{23})$
with $\{\bs{v}_1,\bs{v}_2\}$ linearly independent. Define
\[
    \rho_i  =   \sqrt{v_{i,1}^2 + v_{i,2}^2 + v_{i,3}^2}
\]
to be the length of $\bs{v}_i$ for $i=1,2$ and let $\mu\in[0,\pi)$ denote the angle between
$\bs{v}_1$ and $\bs{v}_2$,
\[
    \cos\mu =   \frac{\langle\bs{v}_1,\bs{v}_2\rangle}{\rho_1\rho_2}.
\]
To describe the change of basis in Section~\ref{subsec:CompCOB1}, let us for brevity use the simplified notation
\[
    (v_{11}, v_{12}, v_{13}, v_{21}, v_{22}, v_{23})
    \mapsto
    (w_{11}, \theta, \phi, w_{21}, w_{22}, \psi),
\]
where, in the notation in Section~\ref{subsec:CompCOB1}, $\theta = \theta^{(1,1)}\in(-\pi,\pi]$,
$\phi = \theta^{(1,2)}\in(0,\pi)$, and $\psi = \theta^{(2,1)}\in(-\pi,\pi]$. The above transformation is given by three rotations,
\[
    R_\theta    =   \begin{pmatrix} 1 & 0 & 0 \\ 0 & \cos\theta & -\sin\theta \\ 0 & \sin\theta & \cos\theta \end{pmatrix},
    \quad
    R_\phi    =   \begin{pmatrix} \cos\phi & -\sin\phi & 0 \\ \sin\phi & \cos\phi & 0 \\ 0 & 0 & 1\end{pmatrix},
    \quad
    R_\psi    =   \begin{pmatrix} 1 & 0 & 0 \\ 0 & \cos\psi & -\sin\psi \\ 0 & \sin\psi & \cos\psi \end{pmatrix},
\]
chosen such that the third coordinate of $R_{-\theta}\bs{v}_1$ is zero and the second coordinate is positive,
the second coordinate of $R_{-\phi} R_{-\theta}\bs{v}_1$ is zero and the first is positive,
and the third coordinate of $R_{-\psi} R_{-\phi} R_{-\theta}\bs{v}_2$ is zero and the second is positive.
Then, letting $R = R_\theta R_\phi R_\psi$, we have $R^{-1}\bs{v}_1 = \bs{w}_1 = (w_{1,1},0,0)$ and
$R^{-1}\bs{v}_2 = \bs{w}_2 = (w_{2,1},w_{2,2},0)$, where
\[
    w_{1,1} =  \rho_1,
    \quad
    w_{2,1} =  \rho_2 \cos\mu,
    \quad
    w_{2,2} =  \rho_2 \sin\mu.
\]
From
\[
R
=
\begin{pmatrix}
 \cos \phi & -\sin \phi \cos \psi & \sin \phi \sin \psi \\
 \cos \theta \sin \phi & \cos \theta \cos \phi \cos \psi-\sin \theta
   \sin \psi & -\sin \theta \cos \psi-\cos \theta \cos \phi \sin
   \psi \\
 \sin \theta \sin \phi & \sin \theta \cos \phi \cos \psi+\cos \theta
   \sin \psi & \cos \theta \cos \psi-\sin \theta \cos \phi \sin \psi
\end{pmatrix},
\]
one computes that
\[
    dv_{1,1}\wedge dv_{1,2}\wedge dv_{1,3}\wedge dv_{2,1}\wedge dv_{2,2}\wedge dv_{2,3}
    =
    w_{1,1}^2 w_{2,2} \sin\phi \;
    dw_{1,1}\wedge d\phi \wedge d\theta \wedge dw_{2,1} \wedge dw_{2,2} \wedge d\psi.
\]
The integral over the $\OO_3$-coordinates $(\theta,\phi,\psi)$, corresponding to the integral over $\Theta$ in
Equation~\eqref{eq:COB1Integral}, is
\[
    \int_{-\pi}^\pi \int_{0}^\pi \int_{-\pi}^\pi \sin\phi \; d\theta d\phi d\psi
    =
    8\pi^2.
\]
That is, if $f$ is an integrable $\OO_3$-invariant smooth function in the $v_{i,j}$, then
\[
    \int_{\R^{6}} f(\bs{v})\;dv_{1,1}\wedge dv_{1,2}\wedge dv_{1,3}\wedge dv_{2,1}\wedge dv_{2,2}\wedge dv_{2,3}
    =
    8\pi^2 \int_W f(\bs{w}) w_{1,1}^2 w_{2,2} \; dw_{1,1}\wedge dw_{2,1} \wedge dw_{2,2},
\]
where $W = \{ (w_{1,1}, w_{2,1}, w_{2,2})\in\R^3 \mid w_{1,1}, w_{2,2} > 0\}$.

For the second change of variables transforming the $w_{i,j}$ into the invariants $u_{i,j}$, note that
\[
    u_{1,1} =   w_{1,1}^2,
    \quad
    u_{2,1} =   w_{1,1}w_{2,1},
    \quad
    u_{2,2} =   w_{2,1}^2 + w_{2,2}^2
\]
i.e.,
\[
    w_{1,1} =   \sqrt{u_{1,1}},
    \quad
    w_{2,1} =   \frac{u_{2,1}}{\sqrt{u_{1,1}}},
    \quad
    w_{2,2} =   \frac{\sqrt{u_{1,1}u_{2,2} - u_{2,1}^2}}{\sqrt{u_{1,1}}}.
\]
Then
\[
    \frac{\partial w_{1,1}}{\partial u_{1,1}}
    =
    \frac{1}{2\sqrt{u_{1,1}}},
    \quad
    \frac{\partial w_{2,1}}{\partial u_{2,1}}
    =
    \frac{1}{\sqrt{u_{1,1}}},
    \quad
    \frac{\partial w_{2,2}}{\partial u_{2,2}}
    =
    \frac{\sqrt{u_{1,1}}}{2\sqrt{u_{1,1}u_{2,2} - u_{2,1}^2}},
\]
so that
\[
    \left|\frac{\partial\bs{w}}{\partial\bs{u}}\right|
    =
    \frac{1}{4\sqrt{u_{1,1}(u_{1,1}u_{2,2} - u_{2,1}^2)}}.
\]
Then we have
\[
    8\pi^2 w_{1,1}^2 w_{2,2} \; dw_{1,1}\wedge dw_{2,1} \wedge dw_{2,2}
    =
    2\pi^2 \; du_{1,1}\wedge du_{2,1}\wedge du_{2,2},
\]
i.e., $\lambda_{2,3} = 2\pi^2$ is constant.
With $f$ as above and $F$ such that $F(\bs{u}) = f(\bs{v})$,
\[
    \int_{\R^{6}} f(\bs{v})\;dv_{1,1}\wedge dv_{1,2}\wedge dv_{1,3}\wedge dv_{2,1}\wedge dv_{2,2}\wedge dv_{2,3}
    =
    2\pi^2 \int_X F(\bs{u}) \; du_{1,1}\wedge du_{2,1} \wedge du_{2,2}.
\]
Using Corollary~\ref{cor:HilbImage}, one computes that the image $X$ of the Hilbert embedding is given by
$u_{1,1}\geq 0$, $u_{2,2}\geq 0$, and $u_{2,1}^2\leq u_{1,1}u_{2,2}$.

\subsection{Two particles in $\R^2$}
\label{subsec:Ex2inR2}

The case of two particles in $\R^2$, i.e., $k = m = 2$, is similar to the previous example but illustrates
the behavior when $k = n$. Let $\bs{v}_1 = (v_{11}, v_{12})$ and $\bs{v}_2 = (v_{21}, v_{22})$
with $\{\bs{v}_1,\bs{v}_2\}$ linearly independent. We once again define the length
\[
    \rho_i  =   \sqrt{v_{i,1}^2 + v_{i,2}^2}
\]
for $i=1,2$ and let $\mu\in[0,\pi)$ be the angle between
$\bs{v}_1$ and $\bs{v}_2$,
\[
    \cos\mu =   \frac{\langle\bs{v}_1,\bs{v}_2\rangle}{\rho_1\rho_2}.
\]
Then the change of basis in Section~\ref{subsec:CompCOB1} is of the form
\[
    (v_{11}, v_{12}, v_{21}, v_{22})
    \mapsto
    (w_{11}, \theta, w_{21}, w_{22}).
\]
This transformation is given by a single rotation
\[
    R_\theta    =   \begin{pmatrix} \cos\theta & -\sin\theta \\ \sin\theta & \cos\theta \end{pmatrix},
\]
chosen such that the second coordinate of $R_{-\theta}\bs{v}_1 = \bs{w}_1$ is zero and the first coordinate is positive.
Then the second coordinate of $R_{-\theta}\bs{v}_2 = \bs{w}_2$ can be assumed positive via the reflection fixing the first coordinate,
introducing a factor of $2$. We again have
\[
    w_{1,1} =  \rho_1,
    \quad
    w_{2,1} =  \rho_2 \cos\mu,
    \quad
    w_{2,3} =  \rho_2 \sin\mu.
\]
One computes that
\[
    dv_{1,1}\wedge dv_{1,2}\wedge dv_{2,1}\wedge dv_{2,2}
    =
    2 w_{1,1} dw_{1,1}\wedge d\theta \wedge dw_{2,1} \wedge dw_{2,2},
\]
and the integral over the $\OO_2$-coordinate $\theta$, that over $\Theta$ in Equation~\ref{eq:COB1Integral}, is simply
\[
    \int_{-\pi}^\pi \; d\theta
    =
    2\pi.
\]
Then if $f$ is an integrable $\OO_2$-invariant smooth function in the $v_{i,j}$,
\[
    \int_{\R^{4}} f(\bs{v})\;dv_{1,1}\wedge dv_{1,2}\wedge dv_{2,1}\wedge dv_{2,2}
    =
    4\pi \int_W f(\bs{w}) w_{1,1} \; dw_{1,1}\wedge dw_{2,1} \wedge dw_{2,2},
\]
where $W = \{ (w_{1,1}, w_{2,1}, w_{2,2})\in\R^3 \mid w_{1,1}, w_{2,2} > 0\}$.

The second change of variables transforming the $w_{i,j}$ into the invariants $u_{i,j}$ is identical to that in Section~\ref{subsec:Ex2inR3}
so that
\[
    4\pi w_{1,1} \; dw_{1,1}\wedge dw_{2,1} \wedge dw_{2,2}
    =
    \frac{\pi}{\sqrt{u_{1,1}u_{2,2} - u_{2,1}^2}} \; du_{1,1}\wedge du_{2,1}\wedge du_{2,2},
\]
i.e.,
\[
    \lambda_{2,2} = \frac{\pi}{\sqrt{u_{1,1}u_{2,2} - u_{2,1}^2}}.
\]
Note that $\lambda_{2,2}$ is singular when $|G_2| = u_{1,1}u_{2,2} - u_{2,1}^2 = 0$, i.e., off the principal orbit type.
With $f$ as above and $F$ such that $F(\bs{u}) = f(\bs{v})$,
\[
    \int_{\R^{4}} f(\bs{v})\;dv_{1,1}\wedge dv_{1,2}\wedge dv_{2,1}\wedge dv_{2,2}
    =
    \pi \int_X \frac{F(\bs{u})}{\sqrt{u_{1,1}u_{2,2} - u_{2,1}^2}} \; du_{1,1}\wedge du_{2,1} \wedge du_{2,2}.
\]
By Corollary~\ref{cor:HilbImage}, we again have that $X$ is given by
$u_{1,1}\geq 0$, $u_{2,2}\geq 0$, and $u_{2,1}^2\leq u_{1,1}u_{2,2}$.


\bibliographystyle{amsplain}
\bibliography{HSW-measureOm.bib}
\end{document}